\documentclass[11pt]{amsart}
\usepackage{amsmath,amssymb,xcolor}
\newtheorem{theorem}{Theorem}
\newtheorem{proposition}[theorem]{Proposition}

\theoremstyle{definition}
\newtheorem{remark}[theorem]{Remark}

\newcommand{\Graph}{{\rm Graph}}

\newcommand{\bS}{{\mathbf S}}
\newcommand{\spt}{\mathop{\rm spt}}
\newcommand{\vol}{{\rm vol}}

\begin{document}
\title[Apriori estimates for optimal transport maps]{A geometric approach to apriori estimates for optimal transport maps}

\thanks{SB's research is supported by the National Science Foundation under grant DMS-2103573 and by the Simons Foundation. RM's research is supported in part by the Canada Research Chairs program CRC-2020-00289, a grant from the Simons Foundation (923125, McCann), and Natural Sciences and Engineering Research Council of Canada Discovery Grant RGPIN- 2020--04162.
\copyright \today\ by the authors.}

\author{Simon Brendle}
\address{Columbia University, 2290 Broadway, New York NY 10027, USA} \email{simon.brendle@columbia.edu}
\author{Flavien L\'eger}
\address{INRIA Paris, France} 
\email{flavien.leger@inria.fr}
\author{Robert J. McCann}
\address{Department of Mathematics, University of Toronto, Toronto, Ontario, Canada M5S 2E4} \email{mccann@math.toronto.edu}
\author{Cale Rankin}
\address{
Fields Institute for the Mathematical Sciences and Department of Mathematics, University of Toronto, Toronto, Ontario, Canada M5S 2E4.
\\ {\em Current address:} School of Mathematics, Monash University, 9 Rainforest Walk, Melbourne, VIC 3800, Australia} \email{
cale.rankin@monash.edu}

\begin{abstract}
A key inequality which underpins the regularity theory of optimal transport for costs satisfying the Ma--Trudinger--Wang condition
is the Pogorelov second derivative bound.  This translates to an apriori  interior $C^1$ estimate for smooth optimal maps.
Here we give a new derivation of this estimate which relies in part on Kim, McCann and Warren's observation that the graph of an optimal map becomes a volume maximizing spacelike submanifold when the product of the source and target domains is endowed with a suitable pseudo-Riemannian geometry that combines both the marginal densities and the cost.

  \medskip


\end{abstract}

\maketitle

\section{Introduction}

Apriori second-derivative estimates for solutions of the Monge--Amp\`ere equation date back to Pogorelov \cite{Pogorelov64,Pogorelov71}, who combined them 
with a third-derivative estimate of Calabi \cite{Calabi58} to deduce the regularity of solutions to Minkowski's problem. Such estimates have a rich history \cite{TrudingerUrbas84,TrudingerWang08}. 
They play a key role in the regularity theory of optimal transportation,  which involve Monge--Amp\`ere type equations whose exact form depends on the cost optimized.
For the quadratic cost function,  this theory was developed by Delano\"e \cite{Delanoe91}
in the plane and by Urbas \cite{Urbas97} in higher dimensions, parallel results being obtained using different techniques by Caffarelli \cite{Caffarelli92,Caffarelli92b,Caffarelli96b}. At the same time, a regularity theory for reflector antenna design was developed by X.-J.~Wang \cite{Wang96}, which can also be seen as an optimal transport problem with the restriction of a logarithmic cost to the sphere \cite{Wang04}.  For more general cost functions,  Ma, Trudinger and Wang \cite{MaTrudingerWang05,TrudingerWang09b} identified a sufficient condition for the regularity of optimal maps, whose necessity was deduced by Loeper \cite{Loeper09}.  Although challenging to interpret,  Ma, Trudinger and Wang's condition (A3) (and its weak variant (A3w)) were subsequently understood by Kim and McCann
as the positivity (respectively non-negativity) of certain sectional curvatures of a pseudo-Riemannian metric induced on the product of the source and target domains by the transportation cost~\cite{KimMcCann10}.  With Warren~\cite{KimMcCannWarren10}, Kim and McCann discovered that the graph of an optimal map becomes a volume maximizing spacelike submanifold when the same metric is conformally rescaled so that its volume gives the product of the source and target densities.  
 
The purpose of the present article is to exploit this geometric perspective 
 to give a new derivation of Pogorelov type apriori 
estimates for optimal maps with respect to costs satisfying the Ma--Trudinger--Wang condition.  Our approach is different but related 
to the strategy developed for $m=2=n$ by Warren \cite{Warren-unpublished}, who was indirectly inspired \cite{LiSalavessa11} 
by the same sources \cite{Wang02,TsuiWang04} as us. It 
yields new insights,  which we hope may lead to further progress: in particular, it allows us to separate the general structure
of the maximum principle estimate from the particulars of the application.  

Although it will not be shown here, interior Pogorelov estimates persist even under the weaker variant (A3w) of the Ma--Trudinger--Wang condition, 
but the argument is more involved~\cite{LiuTrudingerWang10};
regularity of the optimal transport map follows from appropriate domain convexity conditions \cite{TrudingerWang09b}. 
For the quadratic cost,  $2$-uniformity of the convexity required \cite{Caffarelli96b,Urbas97,TrudingerWang09b}
was recently relaxed by Chen, Liu and Wang \cite{ChenLiuWang19,ChenLiuWang21}, though some smoothness remains necessary \cite{Jhaveri19}. Finally,  it is interesting to note that for the quadratic cost,  Kim and McCann's metric \cite{KimMcCann10} restricted to the graph of optimal map coincides with the Hessian metric of the solution of the Monge--Amp\`ere equation,  introduced by Calabi \cite{Calabi58}.

The next section of this note gives a general estimate for maximal spacelike submanifolds  in a fixed pseudo-Riemannian background. The last section explains how to use this to recover and apply
Ma, Trudinger and Wang's local $C^2$ estimate for optimal maps.

\section{Estimates for maximal spacelike submanifolds}

Our ambient space will be a fixed manifold $\hat{M}$ of dimension $n+m$. Let $\hat{g}$ be a pseudo-metric on $\hat{M}$ of signature $(n,m)$. Let $\hat{D}$ denote the Levi-Civita connection on $\hat{M}$, and let $\hat{R}$ denote its Riemann curvature tensor. We adopt the sign convention where 
\[\hat{R}(X,Y,Z,W) = -\hat{g}(\hat{D}_X \hat{D}_Y Z - \hat{D}_Y \hat{D}_X Z - \hat{D}_{[X,Y]} Z,W)\] 
for all vector fields $X,Y,Z,W$ on $\hat{M}$. Let $\hat{E}_1,\hdots,\hat{E}_{n+m}$ denote a local frame of vector fields on $\hat{M}$ which are linearly independent at each point. At each point on $\hat{M}$, the matrix $\hat{g}(\hat{E}_\alpha,\hat{E}_\beta)$ is invertible. We denote its inverse by $\hat{\sigma}^{\alpha\beta}$. We will also fix a smooth symmetric $(0,2)$-tensor field $\hat{S}$ on $\hat{M}$.

We are interested in studying $n$-dimensional spacelike submanifolds of $\hat{M}$. Given such a submanifold $M$, let $g$ and $S$ denote the respective restrictions of $\hat{g}$ and $\hat{S}$ to $M$.  Since $M$ is spacelike, $g$ is positive definite. As in \cite{ONeill}, let $D$ denote the Levi-Civita connection on $M$, and let $D^\perp$ denote the connection on the normal bundle of $M$. We denote by $I\!I: TM \times TM \to T^\perp M$ the second fundamental form of $M$. In other words, if $X$ and $Y$ are two tangential vector fields on $M$, then 
\[I\!I(X,Y) = \hat{D}_X Y - D_X Y.\] 
The mean curvature vector of $M$ is defined as the trace of $I\!I$, which we assume vanishes identically.

We will use the notation  
\[(\hat{D}_V I\!I)(X,Y) = \hat{D}_V(I\!I(X,Y)) - I\!I(D_V X,Y) - I\!I(X,D_V Y)\]
and 
\[(D_V^\perp I\!I)(X,Y) = D_V^\perp(I\!I(X,Y)) - I\!I(D_V X,Y) - I\!I(X,D_V Y)\] 
for all tangential vector fields $X,Y,V$ on $M$ (compare \cite{ONeill}, p.~114). Both expressions are tensorial in $X,Y,V$. Moreover, $(D_V^\perp I\!I)(X,Y) \in T^\perp M$ is the normal component of $(\hat{D}_V I\!I)(X,Y) \in T\hat{M}$.

In the following, we assume that $\{e_1,\hdots,e_n\}$ is a local orthonormal frame on $M$ and $\{e_1^\perp,\hdots,e_m^\perp\}$ is a local orthonormal frame for the normal bundle of $M$. Then $\hat{g}(e_k,e_l) = \delta_{kl}$, $\hat{g}(e_p^\perp,e_q^\perp) = -\delta_{pq}$, and $\hat{g}(e_k,e_p^\perp) = 0$. 

Given $(\hat{M},\hat{g},\hat{S})$, we consider an $n$-dimensional spacelike submanifold $M \subset \hat{M}$ with zero mean curvature. We will show that the restriction $S$ of  $\hat{S}$ to $M$ satisfies an elliptic partial differential equation with coefficients and inhomogeneities given by $\hat{g}$ and $\hat{S}$ and their first two derivatives. This identity contains quadratic terms in the second fundamental form. Since $I\!I(X,Y) \in T^\perp M$ and the metric on the normal space $T^\perp M$ is negative definite, these quadratic terms have a favorable sign, allowing us to apply the maximum principle. This calculation is inspired in part by identities found for hypersurfaces by Calabi \cite{Calabi70}, and in the split geometry of the next section by Mealy \cite{Mealy91}, Harvey and Lawson \cite{HarveyLawson12}.

\begin{proposition}[An elliptic identity for zero mean curvature submanifolds]
\label{pde.for.S}
Let $\hat{R}$ denote the Riemann curvature tensor and $\hat{D}$ the Levi-Civita connection of a signature $(n,m)$ pseudo-Riemannian manifold $(\hat{M},\hat{g})$.
Let $M \subset \hat{M}$ be an $n$-dimensional spacelike submanifold whose second fundamental form $I\!I: TM \times TM \to T^\perp M$ has vanishing trace.
Let $\hat{S}$ be a symmetric $(0,2)$-tensor field on $\hat{M}$, and let $S$ denote its restriction to $M$. Let $\hat{E}_1,\ldots,\hat{E}_{n+m}$ denote a local frame on $\hat{M}$, and $e_1,\ldots,e_n$ denote a local orthonormal frame on $M$. Finally, let $X,Y$ be tangential vector fields on $M$. Then 
\begin{align*} 
&(\Delta S)(X,Y) \\ 
&= \sum_{l=1}^n (\hat{D}_{e_l,e_l}^2 \hat{S})(X,Y) \\ 
&+ 2 \sum_{l=1}^n (\hat{D}_{e_l} \hat{S})(I\!I(e_l,X),Y) + 2 \sum_{l=1}^n (\hat{D}_{e_l} \hat{S})(X,I\!I(e_l,Y)) \\ 
&+ 2 \sum_{l=1}^n \hat{S}(I\!I(e_l,X),I\!I(e_l,Y)) \\ 
&- \sum_{k,l=1}^n \hat{g}(I\!I(e_l,X),I\!I(e_l,e_k)) \, S(e_k,Y) - \sum_{k,l=1}^n \hat{g}(I\!I(e_l,Y),I\!I(e_l,e_k)) \, S(X,e_k) \\ 
&- \sum_{\alpha{, \beta}=1}^{n+m} \sum_{l=1}^n \hat{\sigma}^{\alpha\beta} \, \hat{R}(e_l,X,e_l,\hat{E}_\alpha) \, \hat{S}(\hat{E}_\beta,Y) 
- \sum_{\alpha{,\beta}=1}^{n+m} \sum_{l=1}^n \hat{\sigma}^{\alpha\beta} \, \hat{R}(e_l,Y,e_l,\hat{E}_\alpha) \, \hat{S}(X,\hat{E}_\beta) \\ 
&+ \sum_{k=1}^n \sum_{l=1}^n \hat{R}(e_l,X,e_l,e_k) \, S(e_k,Y) + \sum_{k=1}^n \sum_{l=1}^n \hat{R}(e_l,Y,e_l,e_k) \, S(X,e_k) 
\end{align*} 
at each point on $M$, where $\hat{\sigma}^{\alpha \beta}$ denotes the inverse of the matrix $\hat{g}(\hat{E}_\alpha,\hat{E}_\beta)$.
\end{proposition}

\begin{proof}
In the following, $X,Y,V,W$ will denote tangential vector fields on $M$. By definition,  
\[S(X,Y) = \hat{S}(X,Y).\] 
Differentiating this identity in direction $W$ gives 
\[(D_W S)(X,Y) = (\hat{D}_W \hat{S})(X,Y) + \hat{S}(I\!I(W,X),Y) + \hat{S}(X,I\!I(W,Y)).\] 
Differentiating this identity in direction $V$, we obtain 
\begin{align*} 
(D_{V,W}^2 S)(X,Y) 
&= (\hat{D}_{V,W}^2 \hat{S})(X,Y) + (\hat{D}_{I\!I(V,W)} \hat{S})(X,Y) \\ 
&+ (\hat{D}_W \hat{S})(I\!I(V,X),Y) + (\hat{D}_W \hat{S})(X,I\!I(V,Y)) \\ 
&+ (\hat{D}_V \hat{S})(I\!I(W,X),Y) + (\hat{D}_V \hat{S})(X,I\!I(W,Y)) \\ 
&+ \hat{S}(I\!I(W,X),I\!I(V,Y)) + \hat{S}(I\!I(V,X),I\!I(W,Y)) \\ 
&+ \hat{S}((\hat{D}_V I\!I)(W,X),Y) + \hat{S}(X,(\hat{D}_V I\!I)(W,Y)) 
\end{align*} 
at each point on $M$. We decompose the term $(\hat{D}_V I\!I)(W,X)$ into its tangential and normal components: 
\begin{equation} 
\label{covariant.derivative.of.II}
(\hat{D}_V I\!I)(W,X) = (D_V^\perp I\!I)(W,X) + \sum_{k=1}^n \hat{g}((\hat{D}_V I\!I)(W,X),e_k) \, e_k. 
\end{equation} 
Using the Codazzi equation, we obtain  
\begin{equation} 
\label{covariant.derivative.of.II.normal.component}
(D_V^\perp I\!I)(W,X) = (D_X^\perp I\!I)(V,W) + \sum_{p=1}^m \hat{R}(V,X,W,e_p^\perp) \, e_p^\perp 
\end{equation} 
(see \cite{ONeill}, p.~115). Note that the curvature term on the right hand side in \eqref{covariant.derivative.of.II.normal.component} comes with a plus sign, due to the fact that the metric on the normal space is negative definite. Moreover, differentiating the identity $\hat{g}(I\!I(W,X),e_k)=0$ in direction $V$ gives 
\begin{equation} 
\label{covariant.derivative.of.II.tangential.component}
\hat{g}((\hat{D}_V I\!I)(W,X),e_k) = -\hat{g}(I\!I(W,X),I\!I(V,e_k)). 
\end{equation} 
Substituting \eqref{covariant.derivative.of.II.normal.component} and \eqref{covariant.derivative.of.II.tangential.component} into \eqref{covariant.derivative.of.II}, we obtain 
\begin{align} 
(\hat{D}_V I\!I)(W,X) 
&= (D_X^\perp I\!I)(V,W) - \sum_{k=1}^n \hat{g}(I\!I(W,X),I\!I(V,e_k)) \, e_k \notag \\ 
&+ \sum_{p=1}^m \hat{R}(V,X,W,e_p^\perp) \, e_p^\perp 
\end{align} 
at each point on $M$. Therefore, 
\begin{align*} 
&(D_{V,W}^2 S)(X,Y) \\ 
&= (\hat{D}_{V,W}^2 \hat{S})(X,Y) + (\hat{D}_{I\!I(V,W)} \hat{S})(X,Y) \\ 
&+ (\hat{D}_W \hat{S})(I\!I(V,X),Y) + (\hat{D}_W \hat{S})(X,I\!I(V,Y)) \\ 
&+ (\hat{D}_V \hat{S})(I\!I(W,X),Y) + (\hat{D}_V \hat{S})(X,I\!I(W,Y)) \\ 
&+ \hat{S}(I\!I(W,X),I\!I(V,Y)) + \hat{S}(I\!I(V,X),I\!I(W,Y)) \\ 
&+ \hat{S}((D_X^\perp I\!I)(V,W),Y) + \hat{S}(X,(D_Y^\perp I\!I)(V,W)) \\ 
&- \sum_{k=1}^n \hat{g}(I\!I(W,X),I\!I(V,e_k)) \, S(e_k,Y) - \sum_{k=1}^n \hat{g}(I\!I(W,Y),I\!I(V,e_k)) \, S(X,e_k) \\ 
&+ \sum_{p=1}^m \hat{R}(V,X,W,e_p^\perp) \, \hat{S}(e_p^\perp,Y) + \sum_{p=1}^m \hat{R}(V,Y,W,e_p^\perp) \, \hat{S}(X,e_p^\perp) 
\end{align*} 
at each point on $M$. Next, we take the trace over $V,W$ and use the fact that the mean curvature vanishes to eliminate (derivatives of) the trace of $I\!I$. Thus, 
\begin{align*} 
&(\Delta S)(X,Y) \\ 
&= \sum_{l=1}^n (\hat{D}_{e_l,e_l}^2 \hat{S})(X,Y) \\ 
&+ 2 \sum_{l=1}^n (\hat{D}_{e_l} \hat{S})(I\!I(e_l,X),Y) + 2 \sum_{l=1}^n (\hat{D}_{e_l} \hat{S})(X,I\!I(e_l,Y)) \\ 
&+ 2 \sum_{l=1}^n \hat{S}(I\!I(e_l,X),I\!I(e_l,Y)) \\ 
&- \sum_{k,l=1}^n \hat{g}(I\!I(e_l,X),I\!I(e_l,e_k)) \, S(e_k,Y) - \sum_{k,l=1}^n \hat{g}(I\!I(e_l,Y),I\!I(e_l,e_k)) \, S(X,e_k) \\ 
&+ \sum_{p=1}^m \sum_{l=1}^n \hat{R}(e_l,X,e_l,e_p^\perp) \, \hat{S}(e_p^\perp,Y) + \sum_{p=1}^m \sum_{l=1}^n \hat{R}(e_l,Y,e_l,e_p^\perp) \, \hat{S}(X,e_p^\perp) 
\end{align*} 
at each point on $M$. Using the identity \[\sum_{\alpha,\beta=1}^{n+m} \hat{\sigma}^{\alpha\beta} \, \hat{E}_\alpha \otimes \hat{E}_\beta = \sum_{k=1}^n e_k \otimes e_k - \sum_{p=1}^m e_p^\perp \otimes e_p^\perp,\] 
the assertion follows. 
\end{proof}

In the remainder of this section, we assume that $\hat{S}$ is positive definite. In other words, $\hat{S}$ is a Riemannian metric on the ambient manifold $\hat{M}$. 

\begin{proposition}
\label{pde.for.S.2}
Under the hypotheses of Proposition \ref{pde.for.S}, assume that $\hat{S}$ is a positive definite symmetric $(0,2)$-tensor field on $\hat{M}$, and let $S$ denote its restriction to $M$. Suppose that $p_0$ is a point on $M$ and $\{e_1,\hdots,e_n\} \subset T_{p_0} M$ is an orthonormal basis with respect to $g$ that diagonalizes $S$. Then there exists a constant 
$C=C(\|\hat{g}, \hat{g}^{-1}, \hat{S} \|_{C^2(\{p_0\})})$ such that
\[(\Delta S)(e_n,e_n) \geq 2 \sum_{l=1}^n (\hat{R}(e_l,e_n,e_l,e_n) - C S(e_l,e_l)) \, S(e_n,e_n)\] 
at the point $p_0$.
\end{proposition}

\begin{proof}
Proposition \ref{pde.for.S} implies 
\begin{align*} 
&(\Delta S)(e_n,e_n) \\ 
&= \sum_{l=1}^n (\hat{D}_{e_l,e_l}^2 \hat{S})(e_n,e_n) + 4 \sum_{l=1}^n (\hat{D}_{e_l} \hat{S})(I\!I(e_l,e_n),e_n) \\ 
&+ 2 \sum_{l=1}^n \hat{S}(I\!I(e_l,e_n),I\!I(e_l,e_n)) - 2 \sum_{l=1}^n \hat{g}(I\!I(e_l,e_n),I\!I(e_l,e_n)) \, S(e_n,e_n) \\ 
&- 2 \sum_{\alpha,\beta=1}^{n+m} \sum_{l=1}^n \hat{\sigma}^{\alpha\beta} \, \hat{R}(e_l,e_n,e_l,\hat{E}_\alpha) \, \hat{S}(\hat{E}_\beta,e_n) + 2 \sum_{l=1}^n \hat{R}(e_l,e_n,e_l,e_n) \, S(e_n,e_n) 
\end{align*} 
at the point $p_0$, where $\hat{\sigma}^{\alpha \beta}$ denotes the inverse of the matrix $\hat{g}(\hat{E}_\alpha,\hat{E}_\beta)$. Since the preceding formula does not depend on the choice of the frame $\hat{E}_1,\hdots,\hat{E}_{n+m}$, we may assume that the frame $\hat{E}_1,\hdots,\hat{E}_{n+m}$ is chosen to be orthonormal with respect to the fixed Riemannian metric $\hat{S}$ on ambient space.

Clearly, $I\!I(e_l,e_n)$ is a normal vector for each $l=1,\hdots,n$. Since the metric on the normal space is negative definite, it follows that 
\[\hat{g}(I\!I(e_l,e_n),I\!I(e_l,e_n)) \leq 0\] 
for each $l=1,\hdots,n$. Therefore, 
\[- 2 \sum_{l=1}^n \hat{g}(I\!I(e_l,e_n),I\!I(e_l,e_n)) \, S(e_n,e_n) \geq 0.\] 
We estimate the following multilinear expressions using the fixed Riemannian metric $\hat{S}$ on ambient space (which reduces to $S$ on the tangent space to $M$): 
\[\sum_{l=1}^n (\hat{D}_{e_l,e_l}^2 \hat{S})(e_n,e_n) \geq -C \sum_{l=1}^n S(e_l,e_l) \, S(e_n,e_n),\] 
\[-2 \sum_{\alpha,\beta=1}^{n+m} \sum_{l=1}^n \hat{\sigma}^{\alpha\beta} \, \hat{R}(e_l,e_n,e_l,\hat{E}_\alpha) \, \hat{S}(\hat{E}_\beta,e_n) \geq -C \sum_{l=1}^n S(e_l,e_l) \, S(e_n,e_n),\] 
and 
\begin{align*} 
&4\sum_{l=1}^n (\hat{D}_{e_l} \hat{S})(I\!I(e_l,e_n),e_n) \\ 
&\geq -C \sum_{l=1}^n S(e_l,e_l)^{\frac{1}{2}} \, S(e_n,e_n)^{\frac{1}{2}} \, \hat{S}(I\!I(e_l,e_n),I\!I(e_l,e_n))^{\frac{1}{2}}
\end{align*} 
at the point $p_0$. Putting everything together, we obtain 
\begin{align*} 
(\Delta S)(e_n,e_n) &\geq -C \sum_{l=1}^n S(e_l,e_l) \, S(e_n,e_n) \\ 
&- C \sum_{l=1}^n S(e_l,e_l)^{\frac{1}{2}} \, S(e_n,e_n)^{\frac{1}{2}} \, \hat{S}(I\!I(e_l,e_n),I\!I(e_l,e_n))^{\frac{1}{2}} \\
&+ 2 \sum_{l=1}^n \hat{S}(I\!I(e_l,e_n),I\!I(e_l,e_n)) \\ 
&+ 2 \sum_{l=1}^n \hat{R}(e_l,e_n,e_l,e_n) \, S(e_n,e_n) 
\end{align*} 
at the point $p_0$. The assertion now follows from Young's inequality. 
\end{proof}

In the following, we denote by $\hat{\varphi}$ a nonnegative smooth function on the ambient manifold $\hat{M}$. The function $\hat{\varphi}$ will play the role of a cutoff function. 

\begin{proposition}
\label{maximum.point}
Under the hypotheses of Proposition \ref{pde.for.S}, assume that $\hat{S}$ is a positive definite symmetric $(0,2)$-tensor field on $\hat{M}$, and let $S$ denote its restriction to $M$. Let $\hat{\varphi}$ be a nonnegative smooth function on $\hat{M}$, and let $\varphi$ denote the restriction of $\hat{\varphi}$ to $M$. Suppose that $p_0$ is a point on $M$ with the property that $\varphi(p_0) > 0$ and the largest eigenvalue of $\varphi^{2n-2} S$ with respect to the metric $g$ attains its maximum at the point $p_0$. Suppose that $\{e_1,\hdots,e_n\} \subset T_{p_0} M$ is an orthonormal basis with respect to $g$ that diagonalizes $S$, and $S(e_n,e_n)$ is the largest eigenvalue of $S$ with respect to $g$ at the point $p_0$. Then 
\[\sum_{l=1}^n \hat{R}(e_l,e_n,e_l,e_n) \leq C\varphi^{-2} \, S(e_n,e_n)\] 
at the point {$p_0$. Here}  $C=C(\|\hat{g}, \hat{g}^{-1}, \hat{S}, \hat{\varphi}\|_{C^2(\{p_0\})})$.
\end{proposition}

\begin{proof}
Since the largest eigenvalue of $\varphi^{2n-2} S$ with respect to the metric $g$ attains its maximum at $p_0$, we know that 
\begin{equation} 
\label{first.order.condition}
(D_{e_l} (\varphi^{2n-2} S))(e_n,e_n) = 0 
\end{equation}
for $l=1,\hdots,n$ and 
\begin{equation} 
\label{second.order.condition}
(\Delta (\varphi^{2n-2} S))(e_n,e_n) \leq 0 
\end{equation} 
at the point $p_0$. Combining \eqref{first.order.condition} and \eqref{second.order.condition}, we obtain 
\begin{align} 
\label{main.inequality}
0 
&\geq \varphi^{-2n+2} \, (\Delta (\varphi^{2n-2} S))(e_n,e_n) \notag \\ 
&- (4n-4) \, \varphi^{-2n+1} \sum_{l=1}^n d\varphi(e_l) \, (D_{e_l} (\varphi^{2n-2} S))(e_n,e_n) \notag \\ 
&= (\Delta S)(e_n,e_n) + (2n-2) \, \varphi^{-1} \, (\Delta \varphi) \, S(e_n,e_n) \\ 
&- (2n-2)(2n-1) \, \varphi^{-2} \sum_{l=1}^n (d\varphi(e_l))^2 \, S(e_n,e_n) \notag
\end{align} 
at the point $p_0$. On the other hand, Proposition \ref{pde.for.S.2} gives 
\begin{equation} 
\label{lower.bound.for.Delta.S}
(\Delta S)(e_n,e_n) \geq 2 \sum_{l=1}^n (\hat{R}(e_l,e_n,e_l,e_n)  - C S(e_l,e_l)) \, S(e_n,e_n) 
\end{equation} 
at the point $p_0$. Since the mean curvature of $M$ vanishes, we obtain 
\begin{equation} 
\label{lower.bound.for.Delta.varphi}
\Delta \varphi = \sum_{l=1}^n (\hat{D}^2 \hat{\varphi})(e_l,e_l) \geq -C \sum_{l=1}^n S(e_l,e_l) 
\end{equation} 
Moreover, 
\begin{equation}
\label{upper.bound.for.gradient.varphi} 
\sum_{l=1}^n (d\varphi(e_l))^2 = \sum_{l=1}^n (d\hat{\varphi}(e_l))^2 \leq C \sum_{l=1}^n S(e_l,e_l) 
\end{equation} 
at the point $p_0$. Substituting \eqref{lower.bound.for.Delta.S}, \eqref{lower.bound.for.Delta.varphi}, and \eqref{upper.bound.for.gradient.varphi} into \eqref{main.inequality}, we conclude that 
\[0 \geq 2 \sum_{l=1}^n (\hat{R}(e_l,e_n,e_l,e_n)  - C\varphi^{-2} S(e_l,e_l)) \, S(e_n,e_n)\]
at the point $p_0$. Since $S(e_l,e_l) \leq S(e_n,e_n)$ for $l=1,\hdots,n$, the assertion follows.
\end{proof}

\section{Application to optimal transport}

Let $X$ and $\bar{X}$ be connected orientable manifolds having the same dimension $n$, equipped with nowhere vanishing smooth volume forms $\rho$ and $\bar{\rho}$ satisfying $\int_X \rho=1=\int_{\bar{X}} \bar{\rho}$.  Although we shall often take $\hat{M} = X \times \bar{X}$, it is useful to allow
$\hat{M} \subset X \times \bar{X}$ to be any open domain equipped with
a bounded smooth cost function $c \in C^\infty(\hat{M})$, and extend $c$ to $(X \times \bar{X}) \setminus \hat{M}$ lower semicontinuously (and boundedly). The optimal transportation problem of Kantorovich is to minimize
\begin{equation}
\label{objective}
\int_{X \times \bar{X}} c \, d\gamma
\end{equation}
among joint measures $\gamma \geq 0$ on $X \times \bar{X}$ having $\rho$ and $\bar{\rho}$ as their left and right marginals.
The basic insight gleaned from linear programming duality is the existence of a pair of functions $u \in L^1(X,\rho)$ and $\bar{u} \in L^1(\bar{X},\bar{\rho})$ satisfying
\begin{equation}\label{dual constraint}
u(x) + \bar{u}(\bar{x}) + c(x,\bar{x}) \geq 0 
\quad\mbox{\rm for all}\quad
(x,\bar{x}) \in X \times \bar{X}
\end{equation}
such that $\gamma$ minimizes \eqref{objective} if and only if it vanishes outside the zero set of \eqref{dual constraint}, e.g.~Theorem 5.10 of  \cite{Villani09}.

Monge's version of the same problem is to find a Borel map $F:X \to \bar{X}$ which pulls $\bar{\rho}$ back to $\pm \rho$
(depending on the sign of $\det (-\frac{\partial^2 c}{\partial x^i \partial \bar{x}^j})$),
such that $\gamma=(\text{\rm id} \times F)_\# \rho$ attains the minimum \eqref{objective}.   
We assume that the matrix $\frac{\partial^2 c}{\partial x^i \partial \bar{x}^k} \, dx^i \otimes d\bar{x}^k$ is invertible at each point of $\hat{M}$, which
is hypothesis (A2) of Ma, Trudinger and Wang \cite{MaTrudingerWang05}.

We define a function $\chi: \hat{M} \to (0,\infty)$ so that 
\begin{equation} 
\label{definition.chi}
\chi(x,\bar{x})^n \det \Big ( \frac{\partial^2 c}{\partial x^i \partial \bar{x}^k}(x,\bar{x}) \Big ) \, dx_1 \wedge \hdots dx_n \wedge d\bar{x}_1 \wedge \hdots \wedge d\bar{x}_n = \pm \rho(x) \wedge \bar{\rho}(\bar{x}).
\end{equation} 
Note that $\chi$ is independent of the choice of coordinates. With this understood, 
\begin{equation}\label{KMW metric}
\hat{g} = -\chi \sum_{i,k=1}^n \frac{\partial^2 c}{\partial x^i \partial \bar{x}^k} \, (dx^i \otimes d\bar{x}^k + d\bar{x}^k \otimes dx^i)
\end{equation}
becomes the Kim--McCann--Warren \cite{KimMcCannWarren10}
pseudo-metric on $\hat{M} \subset X \times \bar{X}$. The induced volume form of $\hat{g}$ is given by $\pm \rho \wedge \bar{\rho}$.
In its original (strong) form, the Ma--Trudinger--Wang condition (A3) is equivalent to the condition that 
\begin{equation}\label{MTW}
\hat{R}(\xi \oplus \bar{0},0 \oplus \bar{\xi},\xi \oplus \bar{0},0 \oplus \bar{\xi}) > 0 
\end{equation}
for all points $(x,\bar{x}) \in \hat{M}$ and all nonvanishing tangent vectors $\xi \in T_x X$ and $\bar{\xi} \in T_{\bar{x}} \bar{X}$ satisfying $\hat{g}(\xi \oplus \bar{0},0 \oplus \bar{\xi})=0$. Here, $0 \in T_x X$ and $\bar{0} \in T_{\bar{x}} \bar{X}$ denote zero vectors, see Remark 4.2 of \cite{KimMcCannWarren10}. Note that the condition \eqref{MTW} is not affected by a conformal change of the pseudo-metric $\hat{g}$.

We next define a Riemannian metric $\hat{S}$ on $\hat{M}$. It is convenient to choose the Riemannian metric $\hat{S}$ in a particular way. To explain this, let us fix a Riemannian metric $h$ on $X$ and let $\vol_h$ denote its associated volume form. We define a symmetric $(0,2)$-tensor field $\bar{h}$ on $\hat{M}$ by 
\begin{equation}\label{h metric}
\bar{h} = \chi^2 \sum_{k,l,p,q=1}^n h^{pq} \, \frac{\partial^2 c}{\partial x^p \partial \bar{x}^k} \, \frac{\partial^2 c}{\partial x^q \partial \bar{x}^l} \, d\bar{x}^k \otimes d\bar{x}^l.
\end{equation}
For each point $x \in X$, $\bar{h}$ defines a Riemannian metric on $\bar{X}$. We denote its volume form by $\vol_{\bar{h}}$. In view of \eqref{definition.chi}, the volume forms of $h$ and $\bar{h}$ are related by $\vol_h \wedge \vol_{\bar{h}} = \pm \rho \wedge \bar{\rho}$. With this understood, we define a Riemannian metric $\hat{S}$ on $\hat{M}$ by 
\begin{equation}\label{S metric}
\hat{S} = h + \bar{h}.
\end{equation}
Note that $\hat{S}$ is not a product metric, as $\bar{h}$ depends on $x \in X$. The induced volume form of $\hat{S}$, like $\hat{g}$, is given by $\pm \rho \wedge \bar{\rho}$. 

Suppose that a diffeomorphism $F: X \to \bar{X}$ solves Monge's transport problem, so that the minimizer $\gamma$ of \eqref{objective}
vanishes outside $\Graph(M) := \{(x,F(x)): x \in X\}$. 
Denote by $g$ and $S$ the restrictions of $\hat{g}$ and $\hat{S}$ to $M := \hat{M} \cap \Graph(F)$. 
Kim, McCann and Warren \cite{KimMcCannWarren10} show that $M$ is locally volume maximizing with respect to the metric $\hat{g}$,
hence has zero mean curvature with respect to $\hat{g}$.

\begin{theorem}[Apriori local estimates for optimal diffeomorphisms]
\label{T:main}
Let $\hat{M} \subset X \times \bar{X}$ be an open domain. Suppose that $0 \leq \hat{\varphi} \in C^\infty(\hat{M})$ has compact support $\spt \hat{\varphi} \subset \hat{M}$. Assume the Ma--Trudinger--Wang conditions {\rm (A2)--(A3)} hold on $\hat{M}$. Let $\kappa$ be a positive constant with the property that 
\begin{equation}\label{MTW.uniform}
\hat{R}(\xi \oplus \bar{0}, 0 \oplus \bar{\xi}, \xi \oplus \bar{0}, 0 \oplus \bar{\xi}) \geq \kappa \, h(\xi,\xi) \, \bar{h}(\bar{\xi},\bar{\xi})
\end{equation}
for all points $(x,\bar{x}) \in \spt \hat{\varphi}$ and all tangent vectors $\xi \in T_x X$ and $\bar{\xi} \in T_{\bar{x}} \bar{X}$ satisfying $\hat{g}(\xi \oplus \bar{0},0 \oplus \bar{\xi})=0$. Let $F:X \to \bar{X}$ be a diffeomorphism that minimizes Monge's cost $c$ between $\rho$ and $\bar{\rho}$. Then there is a constant $C$, depending on $\|\hat{g}, \hat{g}^{-1}, \hat{S}, \hat{\varphi} \|_{C^2\left(\spt \hat{\varphi} \right)}$ and $\| \log (\frac{\rho}{\vol_h}) \|_{C^0(\spt \hat{\varphi})}$, such that 
\[\kappa^{n-1} \varphi^{2n-2} S \leq C g\] 
on $M$. Here, $\varphi$ denotes the restriction of $\hat{\varphi}$ to $M=\hat{M} \cap \Graph(F)$, and $g$ and $S$ denote the restrictions of $\hat{g}$ and $\hat{S}$ from \eqref{KMW metric} and \eqref{S metric} to $M$. 
\end{theorem}

\begin{proof}
We consider a point $p_0 \in M$ where the largest eigenvalue of $\varphi^{2n-2} S$ with respect to the metric $g$ attains its maximum. Such a point exists by the smoothness of the objects in question and the fact that $\hat{\varphi}$ has compact support. If $\varphi(p_0) = 0$, the assertion is trivial. Hence, we may assume that $\varphi(p_0) > 0$. Let us write $p_0 = (x_0,F(x_0))$ for some point $x_0 \in X$. Let $(x^1,\hdots,x^n)$ be a local coordinate system on $X$ such that $h_{ij} = \delta_{ij}$ at $x_0$ and let 
$(\bar{x}^1,\hdots,\bar{x}^n)$ denote a local coordinate system on $\bar{X}$ such that $\bar{h}_{kl} = \delta_{kl}$ at $(x_0,F(x_0))$.

Recall there exist $u \in L^1(X,\rho)$ and $\bar{u} \in L^1(\bar{X},\bar{\rho})$ satisfying the inequality \eqref{dual constraint} with equality on $M$:
\begin{equation}
\label{dual constraint=}
u(x) + \bar{u}(F(x)) + c(x,F(x)) =0.
\end{equation}
From smoothness of $c$ at $(x,F(x))$ it follows that $u$ is continuous and semiconvex in a neighbourhood of $x$, as in \cite{McCannGuillen13}.
The first and second-order conditions for \eqref{dual constraint},\eqref{dual constraint=} assert
\begin{equation}
\label{FOC}
\frac{\partial u}{\partial x^i} + \frac{\partial c}{\partial x^i}\Big|_{\bar{x} =F(x)}=0 \qquad i=1,\ldots,n,
\end{equation}
and symmetric nonnegative definiteness of $(B_{ij})_{1\leq i,j \leq n}$ given by
\begin{equation}
\label{SOC} 
B_{ij} := \frac{\partial^2 u}{\partial x^i \partial x^j} + \frac{\partial^2 c}{\partial x^i \partial x^j} \Big |_{\bar{x} = F(x)}. 
\end{equation} 
A priori these conditions hold almost everywhere, but smoothness of $u$ then follows from the hypothesized smoothness of $F$ and \eqref{FOC},
whence both conditions actually hold everywhere. Differentiating \eqref{FOC} yields
\begin{equation}\label{symmetry}
B_{ij} = -\sum_{k=1}^n \frac{\partial^2 c}{\partial x^i \partial \bar{x}^k} \, \frac{\partial F^k}{\partial x^j}.
\end{equation}
Since $F^\# \bar{\rho} = \pm \rho$ by hypothesis, we find
\[\det(B) \, dx^1 \wedge \cdots \wedge dx^n \wedge \bar{\rho} \Big|_{\bar{x} = F(x)} = \pm \det \Big ( \frac{\partial^2 c}{\partial x^i \partial \bar{x}^k} \Big ) \Big|_{\bar{x} = F(x)} \, \rho \wedge d\bar{x}^1 \wedge \cdots \wedge d\bar{x}^n\] 
As (A2) ensures the last term is non-vanishing on $\hat{M}$, nonnegative definiteness of the symmetric matrix $B$ improves to positive definiteness on $M$. 

In the coordinates we have chosen, $\vol_h = \pm dx^1 \wedge \cdots \wedge dx^n$ and $\vol_{\bar{h}} = \pm d\bar{x}^1 \wedge \cdots d\bar{x}^n$ at the point $(x_0,F(x_0))$. Therefore, 
\[\det(B) \, \frac{\bar{\rho}}{\vol_{\bar{h}}} = \pm \det \Big ( \frac{\partial^2 c}{\partial x^i \partial \bar{x}^k} \Big ) \, \, \frac{\rho}{\vol_h}\] 
at the point $(x_0,F(x_0))$. Moreover, \eqref{definition.chi} implies $\chi^n \det (\frac{\partial^2 c}{\partial x^i \partial \bar{x}^k}) = \pm \frac{\rho}{\vol_h} \, \frac{\bar{\rho}}{\vol_{\bar{h}}}$ at the point $(x_0,F(x_0))$. This gives 
\begin{equation}
  \label{determinant.of.B}
  \chi^n \det(B) = \Big ( \frac{\rho}{\vol_h} \Big )^2
\end{equation}
at the point $x_0$. Hence, if we put $A := \chi B$, then 
\begin{equation}
\label{determinant.of.A}
\det(A) = \Big ( \frac{\rho}{\vol_h} \Big )^2
\end{equation} 
at the point $x_0$. In the next step, we relate the eigenvalues of $A$ with respect to $h$ to the eigenvalues of $S$ with respect to $g$. By a suitable choice of the coordinates $(x^1,\ldots,x^n)$, we can arrange that $A_{ij} = \lambda_i \, \delta_{ij}$ as well as $h_{ij}=\delta_{ij}$ at $x_0$. We define 
\begin{equation}\label{unit.xi}
\xi_i = \frac{\partial}{\partial x_i} \in T_{x_0} X, \qquad \bar{\xi}_i = \lambda_i^{-1} \sum_{k=1}^n \frac{\partial F^k}{\partial x^i} \, \frac{\partial}{\partial \bar{x}^k} \in T_{F(x_0)} \bar{X}
\end{equation}
for $i=1,\hdots,n$. Then
\begin{equation} 
\hat{g}(\xi_i \oplus \bar{0},0 \oplus \bar{\xi}_j) = -\lambda_j^{-1} \, \chi \sum_{k=1}^n \frac{\partial^2 c}{\partial x^i \partial \bar{x}^k} \, \frac{\partial F^k}{\partial x^j} = \lambda_j^{-1} \, A_{ij} = \delta_{ij} 
\end{equation} 
for $i,j=1,\hdots,n$. Moreover, $h(\xi_i,\xi_j) = \delta_{ij}$ and 
\begin{align}
\label{unit.bar.xi}
\bar{h}(\bar{\xi}_i,\bar{\xi}_j) 
&= \lambda_i^{-1} \, \lambda_j^{-1} \, \chi^2 \sum_{k,l,p,q=1}^n h^{pq} \, \frac{\partial^2 c}{\partial x^p \partial \bar{x}^k} \, \frac{\partial^2 c}{\partial x^q \partial \bar{x}^l} \, \frac{\partial F^k}{\partial x^i} \, \frac{\partial F^l}{\partial x^j} \notag \\ 
&= \lambda_i^{-1} \, \lambda_j^{-1} \sum_{p,q=1}^n h^{pq} \, A_{pi} \, A_{qj} \\ 
&= \delta_{ij} \notag
\end{align}
for $i,j=1,\hdots,n$. We next define 
\[e_i = (2\lambda_i)^{-\frac{1}{2}} \, (\xi_i \oplus \lambda_i \, \bar{\xi}_i) \in T_{(x_0,F(x_0))} M\] 
and 
\[\mu_i = \frac{1}{2} (\lambda_i+\lambda_i^{-1})\] 
for $i=1,\hdots,n$. Then 
\[\hat{g}(e_i,e_j) = (2\lambda_i)^{-\frac{1}{2}} \, (2\lambda_j)^{-\frac{1}{2}} \, (\lambda_j \, \hat{g}(\xi_i \oplus \bar{0},0 \oplus \bar{\xi}_j) + \lambda_i \, \hat{g}(0 \oplus \bar{\xi}_i,\xi_j \oplus \bar{0})) = \delta_{ij}\] 
and 
\[S(e_i,e_j) = (2\lambda_i)^{-\frac{1}{2}} \, (2\lambda_j)^{-\frac{1}{2}} \, (h(\xi_i,\xi_j) + \lambda_i \lambda_j \, \bar{h}(\bar{\xi}_i,\bar{\xi}_j)) = \mu_i \, \delta_{ij}\] 
for $i,j=1,\hdots,n$. Thus, $\{e_1,\hdots,e_n\} \subset T_{(x_0,F(x_0))} M$ is an orthonormal basis with respect to $g$ which diagonalizes $S$. Moreover, the eigenvalues of $S$ with respect to $g$ at the point $p_0$ are given by $\mu_1,\hdots,\mu_n$. Without loss of generality, we may assume that $\mu_i \leq \mu_n$ for $i=1,\hdots,n-1$, so that $\mu_n$ is the largest eigenvalue of $S$ with respect to $g$ at the point $p_0$.

Note that $\hat{g}(\xi_i \oplus \bar{0},0 \oplus \bar{\xi}_n) = \hat{g}(\xi_n \oplus \bar{0},0 \oplus \bar{\xi}_i) = 0$ for $i=1,\hdots,n-1$. Moreover, it follows from \eqref{unit.xi} and \eqref{unit.bar.xi} that $h(\xi_i,\xi_i) = h(\xi_n,\xi_n) = \bar{h}(\bar{\xi}_i,\bar{\xi}_i) = \bar{h}(\bar{\xi}_n,\bar{\xi}_n) = 1$ for $i=1,\hdots,n-1$. Therefore, the 
uniform Ma--Trudinger--Wang condition \eqref{MTW.uniform} yields 
\[\hat{R}(\xi_i \oplus \bar{0},0 \oplus \bar{\xi}_n,\xi_i \oplus \bar{0},0 \oplus \bar{\xi}_n) \geq \kappa \, h(\xi_i,\xi_i) \, \bar{h}(\bar{\xi}_n,\bar{\xi}_n) = \kappa\] 
and 
\[\hat{R}(\xi_n \oplus \bar{0},0 \oplus \bar{\xi}_i,\xi_n \oplus \bar{0},0 \oplus \bar{\xi}_i) \geq \kappa \, h(\xi_n,\xi_n) \, \bar{h}(\bar{\xi}_i,\bar{\xi}_i) = \kappa\] 
for $i=1,\hdots,n-1$. Using the special structure of the curvature tensor $\hat{R}$ described in Remark 4.2 of \cite{KimMcCannWarren10} and Lemma 4.1 of \cite{KimMcCann10}, we obtain 
\[\hat{R}(\xi_i \oplus \bar{0},\xi_n \oplus \bar{0},\xi_i \oplus \bar{0},\xi_n \oplus \bar{0}) = 0\] 
and 
\[\hat{R}(0 \oplus \bar{\xi}_i,0 \oplus \bar{\xi}_n,0 \oplus \bar{\xi}_i,0 \oplus \bar{\xi}_n) = 0\] 
for $i=1,\hdots,n-1$. Using the identity 
\begin{align*} 
&\hat{R}(\xi_i \oplus \lambda_i \, \bar{\xi}_i,\xi_n \oplus \lambda_n \, \bar{\xi}_n,\xi_i \oplus \lambda_i \, \bar{\xi}_i,\xi_n \oplus \lambda_n \, \bar{\xi}_n) \\ 
&= \hat{R}(\xi_i \oplus \bar{0},\xi_n \oplus \bar{0},\xi_i \oplus \bar{0},\xi_n \oplus \bar{0}) + \lambda_i^2 \lambda_n^2 \, \hat{R}(0 \oplus \bar{\xi}_i,0 \oplus \bar{\xi}_n,0 \oplus \bar{\xi}_i,0 \oplus \bar{\xi}_n) \\ 
&+ \lambda_n^2 \, \hat{R}(\xi_i \oplus \bar{0},0 \oplus \bar{\xi}_n,\xi_i \oplus \bar{0},0 \oplus \bar{\xi}_n) + \lambda_i^2 \, \hat{R}(0 \oplus \bar{\xi}_i,\xi_n \oplus \bar{0},0 \oplus \bar{\xi}_i,\xi_n \oplus \bar{0}) \\ 
&+ 2\lambda_i \, \hat{R}(\xi_i \oplus \bar{0},\xi_n \oplus \bar{0},0 \oplus \bar{\xi}_i,\xi_n \oplus \bar{0}) + 2\lambda_i\lambda_n^2 \, \hat{R}(\xi_i \oplus \bar{0},0 \oplus \bar{\xi}_n,0 \oplus \bar{\xi}_i,0 \oplus \bar{\xi}_n) \\ 
&+ 2\lambda_n \, \hat{R}(\xi_i \oplus \bar{0},\xi_n \oplus \bar{0},\xi_i \oplus \bar{0},0 \oplus \bar{\xi}_n) + 2\lambda_i^2\lambda_n \, \hat{R}(0 \oplus \bar{\xi}_i,\xi_n \oplus \bar{0},0 \oplus \bar{\xi}_i,0 \oplus \bar{\xi}_n) \\ 
&+ 2\lambda_i\lambda_n \, [\hat{R}(\xi_i \oplus \bar{0},\xi_n \oplus \bar{0},0 \oplus \bar{\xi}_i,0 \oplus \bar{\xi}_n) + \hat{R}(\xi_i \oplus \bar{0},0 \oplus \bar{\xi}_n,0 \oplus \bar{\xi}_i,\xi_n \oplus \bar{0})], 
\end{align*} 
we obtain 
\begin{align*} 
&\hat{R}(\xi_i \oplus \lambda_i \, \bar{\xi}_i,\xi_n \oplus \lambda_n \, \bar{\xi}_n,\xi_i \oplus \lambda_i \, \bar{\xi}_i,\xi_n \oplus \lambda_n \, \bar{\xi}_n) \\ 
&\geq \kappa \, (\lambda_n^2+\lambda_i^2) - C \, (\lambda_i + \lambda_i \lambda_n^2) - C \, (\lambda_n + \lambda_i^2 \lambda_n) - C\lambda_i \lambda_n 
\end{align*} 
for $i=1,\hdots,n-1$. This implies 
\[\hat{R}(e_i,e_n,e_i,e_n) \geq \frac{\kappa}{4} \, (\lambda_i^{-1} \lambda_n + \lambda_i \lambda_n^{-1}) - C \, (\lambda_n+\lambda_n^{-1}) - C \, (\lambda_i+\lambda_i^{-1}) - C \] 
for $i=1,\hdots,n-1$. In the next step, we sum over $i=1,\hdots,n-1$. Since $1 \leq \mu_i \leq \mu_n$ for $i=1,\hdots,n-1$, it follows that 
\begin{equation} 
\label{consequence.of.mtw.condition}
\sum_{i=1}^{n-1} \hat{R}(e_i,e_n,e_i,e_n) \geq \frac{\kappa}{4} \sum_{i=1}^{n-1} (\lambda_i^{-1} \lambda_n + \lambda_i \lambda_n^{-1}) - C\mu_n. 
\end{equation}
Using \eqref{determinant.of.A}, we obtain $\prod_{i=1}^n \lambda_i = \det A = \big ( \frac{\rho}{\vol_h} \big )^2$ at $(x_0,F(x_0))$. Using the arithmetic-geometric mean inequality, we obtain 
\begin{equation} 
\label{sum.lambda}
\sum_{i=1}^{n-1} \lambda_i^{-1} \geq (n-1) \, \bigg ( \prod_{i=1}^{n-1} \lambda_i^{-1} \bigg )^{\frac{1}{n-1}} \geq \frac{1}{C} \, \lambda_n^{\frac{1}{n-1}} 
\end{equation} 
and 
\begin{equation}
\label{sum.lambda.inverse} 
\sum_{i=1}^{n-1} \lambda_i \geq (n-1) \, \bigg ( \prod_{i=1}^{n-1} \lambda_i \bigg )^{\frac{1}{n-1}} \geq \frac{1}{C} \, \lambda_n^{-\frac{1}{n-1}}. 
\end{equation} 
Substituting \eqref{sum.lambda} and \eqref{sum.lambda.inverse} into \eqref{consequence.of.mtw.condition} gives 
\begin{equation} 
\label{key.estimate.1}
\sum_{i=1}^{n-1} \hat{R}(e_i,e_n,e_i,e_n) \geq \frac{\kappa}{C} \, \mu_n^{\frac{n}{n-1}} - C\mu_n. 
\end{equation}
On the other hand, Proposition \ref{maximum.point} implies that 
\begin{equation} 
\label{maximum.principle.calculation}
\sum_{i=1}^{n-1} \hat{R}(e_i,e_n,e_i,e_n) \leq C \varphi^{-2} \mu_n 
\end{equation} 
at the point $p_0$. Combining \eqref{key.estimate.1} and \eqref{maximum.principle.calculation}, we conclude that $\kappa \mu_n^{\frac{n}{n-1}} \leq C \varphi^{-2} \mu_n + C\mu_n$ at the point $p_0$. Since $\varphi \leq C$, it follows that $\kappa \varphi^2 \mu_n^{\frac{1}{n-1}} \leq C$ at the point $p_0$. This shows that the largest eigenvalue of $\kappa^{n-1} \varphi^{2n-2} S$ is uniformly bounded from above at the point $p_0$, and hence on all of $\spt \hat{\varphi}$.
\end{proof}

\begin{remark}[Quantitatively spacelike, slope, recovery of $C^1$ bounds] \label{rem:quant}
We may view the apriori lower bound $\kappa^{n-1} \varphi^{2n-2} S \leq Cg$ given by Theorem~\ref{T:main} as quantifying the spacelikeness of $\Graph(F)$. 
Since $\hat{g}$ vanishes on the horizontal fibres $X \times \{\bar{x}\}$ and vertical fibres $\{x\} \times \bar{X}$,  a fortiori this quantifies the transversality with which $F$
intersects these fibres, thus translating to a $C^1$ bound for $F^{-1}$ and for $F$.  More explicitly, the estimate in Theorem \ref{T:main} 
is equivalent to the inequality 
\begin{equation} 
\label{eigenvalue.bound} 
\lambda_i+\lambda_i^{-1} \leq C \kappa^{1-n} \varphi^{2-2n}
\end{equation}
for $i=1,\hdots,n$, where $\lambda_1,\hdots,\lambda_n$ denote the eigenvalues of $A = \chi B$ with respect to the metric $h$. The inequality \eqref{eigenvalue.bound} implies two separate bounds for $\lambda_i$ and for $\lambda_i^{-1}$. The bound for $\lambda_i$ quantifies the transversality with which the graph of $F$ intersects the vertical fibers $\{x\} \times \bar{X}$; it corresponds to an interior $C^1$-estimate for $F$. The bound for $\lambda_i^{-1}$ quantifies the transversality with which the graph of $F$ intersects the horizontal fibers $X \times \{\bar{x}\}$; it corresponds to an interior $C^1$-estimate for $F^{-1}$. In the setting of Ma--Trudinger--Wang, we may take $X,\bar{X}$ as open subsets of $\mathbf{R}^n$ and $h$ as the usual Euclidean metric. The definition \eqref{SOC} of $B$ then implies the Ma--Trudinger--Wang estimates as they appear in \cite{MaTrudingerWang05}.
\end{remark}

\begin{remark}[Optimal maps and dual potentials]\label{R:dual potentials}
The potentials $u \in L^1(X,\rho)$ and $\bar{u} \in L^1(\bar{X},\bar{\rho})$ from \eqref{dual constraint} which certify optimality of $\gamma$ in \eqref{objective} can be interpreted as Lagrange
multipliers for the marginal constraints $\rho$ and $\bar{\rho}$.  When $\hat{M} = X \times \bar{X}$, then
$u$ is well-known to become continuous and semiconvex if the second $x$ derivatives of $c$ are bounded locally on $X$ uniformly with respect to $\bar{x} \in \bar{X}$.
Similarly, $\bar{u}$ is semiconvex if the corresponding condition holds under the interchange $x \leftrightarrow \bar{x}$, e.g.~\cite{McCannGuillen13}.
 A sufficient condition for existence of a Borel solution $F:X \to \bar{X}$ to Monge's problem is then that
$x \in X \mapsto c(x,\bar{x}) - c(x,\bar{x}')$ be free from critical points for all $\bar{x}' \ne \bar{x} \in \bar{X}$.
Under hypothesis {\rm (A1)} of \cite{MaTrudingerWang05}, which asserts that both $c$ and $\tilde c(\bar{x}, x) =c(x,\bar{x})$
satisfy this restriction,  the solution $F$ is injective outside a $\rho$-negligible set.  If, in addition, $F$ happens to be continuously differentiable,
then it follows from non-degeneracy {\rm (A2)} that $F$ is a diffeomorphism, arguing as in the {second} paragraph of the proof of Theorem \ref{T:main}. Higher regularity can be bootstrapped by applying Evans \cite{Evans82} or Krylov \cite{Krylov82} and then standard elliptic theory~\cite{GilbargTrudinger98}.
\end{remark}

\begin{remark}[Regularity of weak solutions]
\label{R:regularity}
Diffeomorphisms $F:X \to \bar{X}$ solving Monge's problem may or may not exist \cite{Caffarelli92,Loeper09},
depending on the choice of volume forms $\rho,\bar{\rho}$ and the cost $c$.
Under an additional hypothesis
-- equivalent to 
geodesic convexity of the vertical fibers $\{x\} \times \bar{X}$ with respect to the Kim-McCann metric \cite{KimMcCann10}, or equivalently the Kim-McCann-Warren metric 
$\hat{g}$ -- Ma, Trudinger and Wang construct a diffeomorphism by starting from a $c$-convex potential $u$ which solves the dual problem and
showing it uniquely solves the associated Monge--Amp\`ere type equation in the sense of Alexandrov.  To establish interior regularity of $u$, Ma, Trudinger, and Wang use the continuity method to construct a smooth solution of the Monge--Amp\`ere type equation which is defined on a small ball and agrees with $u$ on the boundary of the ball.
This involves deforming the boundary data given by $u$ to a new set of boundary data which is smooth and $c$-convex; 
this argument was extended to generated Jacobian equations in \cite{Trudinger14}.  Independently of the regularization parameter,
the $C^2$-estimate ensures the solution of the Monge--Amp\`ere type equation is uniformly elliptic, hence $c$-convex  \cite{TrudingerWang09c,KimMcCann10},  after which higher regularity follows from
Evans--Krylov and standard elliptic theory as in Remark~\ref{R:dual potentials}.
\end{remark}

\begin{remark}[Boundary behaviour, global regularity, compact manifolds] 
Theorem \ref{T:main} only addresses interior regularity. Studying the boundary regularity requires additional techniques, as in e.g.~\cite{Caffarelli92b,Caffarelli96b,Urbas97,TrudingerWang09b,ChenLiuWang21}. 

If $X$ and $\bar{X}$ are compact manifolds without boundary, taking $\hat{M} = X \times \bar{X}$ is inconsistent with global non-degeneracy {\rm (A2)} of $\hat{g}$. 
However, in some situations it is possible to identify an open subset $\hat{M} \subset X \times \bar{X}$, where the conditions {\rm (A2)--(A3)} are satisfied. If we know that $\Graph(F)$ is contained in such a subset $\hat{M}$, then Theorem \ref{T:main} yields a global estimate, extending \cite{Warren-unpublished} to dimensions $n>2$.

As an example, suppose that we take $X$ and $\bar{X}$ to be the round sphere ${\bS}^n$ and define the cost function $c$ by $c(x,\bar{x})=d^2(x,\bar{x})/2$. In this case, we may take $\hat{M}$ to be the set where $c$ is smooth (i.e. the complement of the cut locus). If the densities $\rho$ and $\bar{\rho}$ are bounded from above and below, it is known~\cite{DelanoeLoeper06} that $\Graph(F)$ is disjoint from the cut locus. Regularity then follows \cite{Loeper11}. This line of reasoning can be generalized to perturbations \cite{DelanoeGe10}, submersions \cite{KimMcCann12} or products \cite{FigalliKimMcCannJEMS13} of round spheres.
\end{remark}

\bibliographystyle{plain}

\end{document}